\documentclass[12pt]{amsart}  

\usepackage{bbold}
\usepackage{amscd} 
\usepackage{amsmath,amssymb}
\usepackage{mathtools}
\usepackage{mathrsfs}
\usepackage{color}

\newtheorem{theorem}{Theorem}[section]
\newtheorem{lemma}[theorem]{Lemma}
\newtheorem{corollary}[theorem]{Corollary} 
\theoremstyle{definition}

\newtheorem{remark}[theorem]{Remark} 
\numberwithin{equation}{section}

\numberwithin{equation}{section}

\def\ve{\varepsilon}
\def\be{\begin{equation}}
\def\ee{\end{equation}}
\def\RE{\mathbb R}
\def\CO{{\mathbb C}}
\def\bou{{\mathscr B}}

\def\C{\mathcal C}
\def\B{\mathscr B}

\title[Body resonances for classical waves]{Body resonances for classical waves}

\author{Andrea Mantile}
\author{Andrea Posilicano}
\address{Laboratoire de Math\'{e}matiques de Reims, UMR9008 CNRS et
Universit\'{e} de Reims Champagne-Ardenne, Moulin de la Housse BP 1039, 51687
Reims, France}
\address{DiSAT, Sezione di Matematica, Universit\`a dell'Insubria, via Valleggio 11, I-22100
Como, Italy}
\email{andrea.mantile@univ-reims.fr}
\email{andrea.posilicano@unisubria.it}

\begin{document}

\begin{abstract} We provide a detailed study of the spectral properties of the linear operator $H(\varepsilon)=-(\ve^{2}\chi_{\Omega_{\ve}}+\chi_{\Omega^{c}_{\ve}})\Delta$ modeling, through the wave equation $(\partial_{tt}+H(\ve))u=0$, the dynamics of acoustic waves in the presence of a small inhomogeneity of size $\ve$ having high contrast $\ve^{-2}$. In particular, we give precise results on the localization of the resonances of $H(\ve)$ and their first-order $\ve$-expansions; the latter are explicitly expressed in terms of the eigenvalues and eigenvectors of the Newton potential operator of the set $\Omega$ whose rescaling of size $\ve$ defines $\Omega_{\ve}$. 
\end{abstract}

\maketitle
\section{Introduction}

We consider the $\ve$-dependent linear operator%
\begin{equation}%
\begin{array}
[c]{ccccc}%
H(\varepsilon):H^{2}(\mathbb{R}^{3})\subset L^{2}(\mathbb{R}^{3})\rightarrow
L^{2}(\mathbb{R}^{3})\,, &  & H(\varepsilon):=-b_{\varepsilon}\Delta\,, &  &
b_{\varepsilon}:=\varepsilon^{2}\chi_{\Omega_{\varepsilon}}+\chi
_{\Omega_{\varepsilon}^{c}}\,,
\end{array}
\label{H_eps}%
\end{equation}
where $0<\varepsilon<1$ is a small parameter fixing the size of the
contracted set%
\[
\Omega_{\varepsilon}:=\left\{  x\in\mathbb{R}^{3}:x=\varepsilon y\,,\ y\in
\Omega\right\}  \,,
\]
built from the reference set $\Omega\subset\mathbb{R}^{3}$ which is assumed to be open,
bounded, connected and containing the origin; $\chi_{\Omega_{\varepsilon}}$ is
the characteristic function of $\Omega_{\varepsilon}$ and the contrast
function $b_{\varepsilon}$ in (\ref{H_eps}) models a sharp discontinuity of a
medium across the interface $\partial\Omega_{\varepsilon}$. $H(\varepsilon
)$ generates the dynamics of acoustic pressure waves by the equation%
\begin{equation}
b_{\varepsilon}^{-1}(  x)\,  \frac{\partial^{2}u}{\partial t^{2}}( t, x)
=\Delta u( t, x,t) \,. \label{WE_eps}%
\end{equation}
Here
$\varepsilon=\sqrt{1_{\Omega_{\varepsilon}}b_{\varepsilon}}$ defines the
relative speed of propagation inside $\Omega_{\ve}$; hence,  a small value of $\varepsilon$
corresponds to a regime of small relative speed of propagation. While our work
is focused on the $3D$ case, it is worth mentioning that this equation in 2D 
is also associated to electromagnetic waves propagation in an isotropic medium having unit magnetic permeability and a dielectric permittivity given
by $b_{\varepsilon}$. In this connection, we expect that our analysis may
provide qualitative insights for a wider class of physical systems. It is also
worth mentioning that the stationary (time-harmonic) problem corresponding to
(\ref{WE_eps}) has been considered in recent works as a simplified scattering
model of a dielectric microresonator (see, e.g., \cite{Amm}, \cite{Amm1} and
related references). In this connection, $H(\varepsilon)$ can be considered as
the generator of the dynamics for such specific resonant-scattering models. 
\par
A small-$\ve$ approximation for the solutions of the Cauchy problem corresponding to \eqref{WE_eps} has been
recently discussed in \cite{MaPo PDEA 24}, where a point-scatterer
effective dynamics is validated in the case the initial data are sufficiently regular and supported outside the inhomogeneity $\Omega_{\ve}$. It is important to underline that such
effective dynamics has a rather not-standard definition, in the sense that
it does not corresponds to the one generated by a standard point-perturbation of the Laplacian (as given in \cite[Chapter I.1]{AGHKH}), but
expresses in terms of infinitely-many independent $1D$ dynamical systems, each
fed by the Cauchy data of the problem and by a single eigenvalue of the Newton potential operator of $\Omega$. 
\par
As well known from the theory of second-order Cauchy problems, the wave dynamics generated by $H(\ve)$ is defined in terms of the cosine and sine operators associated to its resolvent
(see \cite[Section 3]{MaPo PDEA 24}). In particular, the computation of
such operators involve integrals of the resolvent along the parabola
$\{  z=(c+i\gamma)^{2}\,,\ \gamma\in\mathbb{R}\}  $ with
arbitrary $c>0$. The analysis developed in \cite{MaPo PDEA 24} requires
resolvent estimates sufficiently far away from the spectrum, depending
on $c>0$, and the main technical issues are concerned with the control of both
the small scale limit and the parabola's infinite branches contributions (involving high-energies resolvent estimates) on a large time. Such analysis provides a
formula, for the dominant part of the wave, involving an infinite sum over
localized states (Green's functions) at energies corresponding to the reciprocals of the 
eigenvalues of the Newtonian operator on $\Omega$. 
The asymptotic formulae presented in \cite{MaPo PDEA 24} suggest that the
dynamical system (\ref{WE_eps}) could be driven, in the small scale regime, by
a (possibly finite) number of resonant states. However, a detailed analysis of
the resonances involved in the dynamics and their behavior as $\ve\ll 1$ have not
been considered yet. 
The aim of this work is the study the generalized
spectral problem for $H(\ve)$, in particular the localization of its
resonances, as $\varepsilon\searrow 0$.

\subsection{Notation} \hfill
\vskip5pt\par\noindent
\noindent$\bullet$ $|\Omega|$ denotes the volume of $\Omega$ and $d_{\Omega}$
is its diameter. 
\vskip5pt\par\noindent
$\bullet$ $\|\cdot\|$ denotes the norm in a space of square integrable functions like $L^{2}(\Omega)$ and $L^{2}(\RE^{3})$ with scalar product $\langle\cdot,\cdot\rangle$; $\|\cdot\|$ also denotes the operator norm for bounded linear operators acting between such spaces. Norms in different spaces are specified with the appropriate subscripts. 
\vskip5pt\par\noindent
$\bullet$ $\bou(X,Y)$ denotes the Banach space of bounded operators from the Banach space $X$ to the Banach space $Y$; $\bou(X,X)\equiv\bou(X)$.
\vskip5pt\par\noindent
$\bullet$ ${\mathfrak{S}}_{\infty }(X,Y)\subseteq \bou(X,Y)$ denotes the set of compact operators.
\vskip5pt\par\noindent
$\bullet$ $H^{2}(\RE^{3})$, denotes the Hilbert-Sobolev space of order $2$, i.e. the space of square integrale functions $f:\RE^{3}\to\CO$ such that $\Delta f$ is square integrable.  
\vskip5pt\par\noindent
$\bullet$ $\mathcal G_{\kappa}$ denotes the kernel function of $(-\Delta-\kappa^{2})^{-1}$, i.e., $\mathcal G_{\kappa}(x)=\frac{e^{i\kappa|x|}}{4\pi|x|}$.
\vskip5pt\par\noindent
$\bullet$ $\varrho(A)$ and $\sigma(A)$ denote the resolvent set and the spectrum of the self-adjoint operator $A$;
$\sigma_{p}(A)$, $\sigma_{disc}(A)$, $\sigma_{ess}(A)$, $\sigma_{cont}(A)$, $\sigma_{ac}(A)$,  $\sigma_{{sc}}(A)$,  denote the point, discrete, essential, continuous, absolutely continuous and singular continuous spectra.
\vskip5pt\par\noindent
$\bullet$ $\CO_{\pm}:=\{z\in\CO:\pm\text{Im}(z)>0\}$.
\vskip5pt\par\noindent
$\bullet$ $D_{r}(z)\subset\CO$ denotes the open disc of radius $r$ and center $z$; $D_{r}\equiv D_{r}(0)$.

\subsection{Results} \hfill
\vskip5pt\par\noindent

We investigated the spectrum and resonances of $H\left(  \varepsilon
\right)  $ in the small-$\ve$ regime. As regards the spectrum, in Lemma \ref{sp} and Corollary
\ref{Prop_empty_sc} we show that, for $\varepsilon\in(  0,1)  $,
one has
\[%
\begin{array}
[c]{ccc}%
\sigma_{ac}\left(  H(  \varepsilon)  \right)  =\left[
0,+\infty\right)  \,, & \text{and} & \sigma_{p}\left(  H\left(  \varepsilon
\right)  \right)  =\sigma_{sc}\left(  H(  \varepsilon)  \right)
=\varnothing\,,
\end{array}
\]
which gives, as regards the resolvent set,  $\varrho(  H(\varepsilon))  :=\mathbb{C}\backslash[  0,+\infty)
$. In particular, $H(  \varepsilon)  $ has no embedded
eigenvalues and in Theorem \ref{Theorem_LAP} we give a limiting absorption
principle for $H(\ve)$: namely, we show that the limits%
\[
\lim_{\delta\searrow0}\,\big(  H(  \varepsilon)
-(\lambda\pm i\delta)\big)  ^{-1}\,,
\]
exist in appropriate weighted spaces for all $\lambda\in\mathbb{R}$. As regards the resonances, after
introducing the operator%
\[
M_{\kappa}(  \varepsilon)  :=1-(  1-\varepsilon^{2})
\kappa^{2}N_{\ve\kappa}\,,\qquad N_{\kappa}u(x):=
\frac1{4\pi}\int_{\Omega}\frac{e^{i\kappa|x-y|}}{|x-y|}\,u(y)\,dy
\,,\quad \kappa\in\CO\,,
\]
using the resolvent formula  \eqref{KR_id} together with \eqref{KR_id_1}, \eqref{S_k_inv_id} and Lemma \ref{Lemma_spectrum}, 
we identify the resonant set of $H(  \varepsilon)  $ with the
exceptional points $\kappa^{2}$ where $M_{\kappa}(  \varepsilon)  $
is not invertible, i.e.%
\begin{equation}
\mathsf{e}(H(\varepsilon))=\{\kappa_{\circ}^{2}(\varepsilon):\ker
(M_{\kappa_{\circ}(\varepsilon)}(\varepsilon))\not =\{0\},\ \kappa_{\circ
}(\varepsilon)\in{\mathbb{C}}_{-}%
\}\,,\label{res_set_def}%
\end{equation}
(see the definition (\ref{e_def}) and the subsequent remark). Then, the
analysis is carried on the corresponding limit operator%
\[
M_{\kappa}\left(  0\right)  :=1-\kappa^{2}N_{0}\,,
\]
whose characteristic points are determined in terms of the spectrum of the
Newton potential operator of the reference set $\Omega$, i.e., 
\[
N_{0}u(x)=\frac1{4\pi}\int_{\Omega}\frac{u(y)\,dy}{|x-y|}\,.
\]
It is well known that $N_{0}$ is
an Hilbert-Schmidt not-negative operator having a $L^{2}(\Omega
)$-complete orthonormal eigensystem $\{(\lambda_{n},e_{n})\}_{1}^{+\infty}$, where the eigenvalues $\lambda_{n}$ accumulate at
zero. Building on these properties and using a small-$\varepsilon$ expansion,
we show that $M_{\kappa}(\varepsilon)$ is invertible depending on the distance
of $\kappa^{2}$ from $\sigma(N_{0}^{-1}))$: in any bounded
region of $\CO$, $\mathsf{e}(H(\varepsilon))$ is contained inside the disjoint union of small disks of radius of order $\ve$ centered at points
$\lambda^{-1}\in\sigma(N_{0}^{-1})$  (see Theorem \ref{loc} and Corollary \ref{coro}). This yields the localization of the resonances as $\varepsilon\searrow0$. Our main result, whose proof is given in Section \ref{Sec_e}, shows that in a small
neighborhood of each $\lambda^{-1}\in\sigma(N_{0}^{-1})$ there are as many
resonances (counted with multiplicity) as the multiplicity of $\lambda$ and
provides their first-order $\ve$-expansions:  
\begin{theorem}
\label{Theorem_main}Let $\lambda_{\circ}\in\sigma_{disc}(N_{0})$ with
$\dim(\ker(\lambda_{\circ}-N_{0}))=m$; let $e_{j}^{\circ}$, $1\leq j\leq m$,
be the corresponding orthonormal eigenvectors. Then, whenever $\varepsilon$ is
sufficiently small, close to $\lambda_{\circ}^{-1}$ there are $m$ (not
necessarily distinct) resonances $\kappa_{\circ,j}^{2}(\varepsilon
)\in\mathsf{e}(H(\varepsilon))$, the functions $\varepsilon\mapsto
\kappa_{\circ,j}^{2}(\varepsilon)$ are analytic in a neighborhood of zero and
\begin{equation}
\kappa_{\circ,j}^{2}(\varepsilon)=\frac{1}{\lambda_{\circ}}-i\varepsilon
\ \frac{|\langle1,e_{j}^{\circ}\rangle|^{2}}{4\pi\lambda_{\circ}^{5/2}%
}+O(\varepsilon^{2})\,.\label{res_exp}%
\end{equation}

\end{theorem}
It is known that the quantitative analysis of the resolvent on the non-physical sheet requires the perturbation theory of eigenvalues of non self-adjoint operators. In this setting, even if the perturbation is analytic, the eigenvalue expansions are given as Puiseux series. However, in our specific case, an implicit function argument allows us to reduce the problem to the perturbative analysis of $N_{\kappa}$ for $\kappa$ close to $0\in\CO$. Since $N_{i\kappa}$ is self-adjoint, an adaptation of the Rellich classical
arguments leads to the power series expansion in \eqref{res_exp}.\par 
The asymptotic analysis of the resonances of $H(\ve)$, in particular \eqref{res_exp}, provides a new insight 
on the small-scale behavior of the system in \eqref{WE_eps}.  As recalled above, an interesting
outcome of \cite{MaPo PDEA 24} is that the main contribution to
the waves dynamics expresses in terms of the eigenvalues in
$\sigma(N_{0}^{-1})$ which, according to (\ref{res_exp}), correspond to the
small-scale limit values of the resonances. This motivates to reconsider the
time-domain problem to enlighten the role of resonances. Since the time
propagator of (\ref{WE_eps}) requires inverse Laplace transforms of the
resolvent of $H(  \varepsilon)  $, we expect that the resolvent
analysis provided here will be a key ingredient in computing resonant-modes
expansions in the asymptotic regime.
\par
The spectral analysis of model operators related to micro-resonators in wave
dynamics has not been investigated so far. Such models are characterized by the enhancement of the scattered field occurring at specific resonant frequencies and, due to their application perspectives, they have attracted an increasing attention in recent years. These resonant phenomena have been shown to emerge
in different physical systems where stationary-waves interact with highly-contrasted small inhomogeneities (see, e.g., \cite{Amm2}). Depending on the
micro-resonator design, the resonant modes (subwavelenght resonances in
physical literature), arise from the excitation of specific spectral points of
the volume potential or of the surface potential involved in the integral
representation of the scattering equation. In the time-harmonic problem
corresponding to our setting, a Lippmann-Schwinger representation of the
scattering allows to express the solution in terms of the Newton potential operator: in
this case, the resonant scattering is determined by volume modes, also
referred to as body-resonances.
\par
It is a commonly accepted idea that subwavelenght resonances should correspond
to resonances in the proper generalized-spectral sense. At this concern, we
remark that the expansions in (\ref{res_exp}) are consistent with the
asymptotic formulae presented in \cite{Amm1} in the time-harmonic case. In
this connection, our results extend the previous analysis, providing the
expected spectral characterization of body resonances.

\section{The model operator and its spectrum}
The linear operator \eqref{H_eps} re-writes as an additive perturbation
of the Laplacian:%
\begin{equation}%
\begin{array}
[c]{ccc}%
H(\varepsilon)=-\Delta+T(\varepsilon)\,, &  & T(\varepsilon):=(  
1-\varepsilon^{2})    \chi_{\Omega_{\varepsilon}}\Delta\,.
\end{array}
\label{H_eps_add}%
\end{equation}
We introduce the free resolvent 
$$
R_{\kappa}\in \mathscr{B}(  L^{2}(  \mathbb{R}%
^{3})  ,H^{2}(  \mathbb{R}^{3})  )\,,\qquad  R_{\kappa} :=(  -\Delta-\kappa^{2})  ^{-1}\,,\qquad \kappa\in\mathbb{C}_{+}\,.
$$
In the following, we will use the equivalence, holding in the Banach space sense,
\be\label{equi}
L^{2}(\RE^{3}\!,\ve):=L^{2}(\RE^{3};b_{\ve}^{-1}dx)\simeq L^{2}(\RE^{3})\,.
\ee
By \cite[Theorem 3.2]{MaPo PDEA 24}, a Rellich-Kato type resolvent formula holds:
\begin{theorem}\label{teoK} $H(\varepsilon)$ is a closed operator in $L^{2}(\mathbb{R}^{3})$ and a non-negative self-adjoint operator in $L^{2}(\RE^{3}\!,\ve)$ with resolvent  
\be\label{res}
R_{\kappa}(\ve):=(H(\varepsilon)-\kappa^{2})^{-1}=R_{\kappa}-R_{\kappa}\Lambda_{\kappa}(    \varepsilon)R_{\kappa}\,,\qquad \kappa\in \CO_{+}\,,
\ee
where
\[
\Lambda_{\kappa}(    \varepsilon)    :=(    1+T(\varepsilon)%
R_{\kappa})    ^{-1}T(\varepsilon)\,, \qquad 
(1+T(\varepsilon)R_{\kappa})  ^{-1}\in{\bou}(  L^{2}(\mathbb{R}^{3}))\,,\qquad \kappa\in \CO_{+}\,.
\]
\end{theorem}
\begin{remark} Let $\kappa\in\CO_{+}$. By \be\label{mani}
R_{\kappa}(  1+T(\varepsilon)R_{\kappa})  ^{-1}
=R_{\kappa}-R_{\kappa}(1+T(\varepsilon)R_{\kappa})^{-1}T(\varepsilon)R_{\kappa}\,,
\ee
one has
\be\label{KR_id}
R_{\kappa}(\ve)=R_{\kappa}(  1+T(\varepsilon)R_{\kappa})  ^{-1}
\ee
and so, by \eqref{res}, there holds
$$
R_{\kappa}(\ve)=R_{\kappa}-R_{\kappa}(\ve)T(\varepsilon)R_{\kappa}\,.
$$
Then, by
\begin{equation}
T(\varepsilon)R_{\kappa}=(    1-\varepsilon^{2})    1_{\Omega_{\varepsilon
}}\Delta R_{\kappa}=-(    1-\varepsilon^{2})    1_{\Omega_{\varepsilon}%
}(    1+\kappa^{2}R_{\kappa})    \,, \label{KR_comp}%
\end{equation}
one gets the relation
\begin{equation}
R_{\kappa}(    \varepsilon) -   R_{\kappa}=(1- \varepsilon^{2})  
R_{\kappa}(    \varepsilon)    1_{\Omega_{\varepsilon}}(  
1+\kappa^{2}R_{\kappa})    \,. \label{Res_id}%
\end{equation}

\end{remark}

\begin{lemma}\label{sp}
$\sigma_{p}(  
H(\varepsilon))    =\varnothing$.
\end{lemma}
\begin{proof}
Since $H(\varepsilon)\ge 0$, it suffices to show that $\sigma_{p}(  
H(\varepsilon)) \cap [0,+\infty)=\varnothing$. Since $b_{\ve}(x)\not=0$ for any $x\in \RE^{3}$, 
$H(\varepsilon)u=0$ is equivalent to $-\Delta u=0$; this gives $0\notin\sigma_{p}(  
H(\varepsilon)) $. Let us now take $\lambda>0$ and let us seek for a solution $u\in H^{2}(  \mathbb{R}^{3})\subset \C(\RE^{3})$ of the equation%
\begin{equation}
(    H(\varepsilon)-\lambda)    u=0\,. \label{eigen_eq}%
\end{equation}
We aim to show that $u=0$. By the definition of $H(\varepsilon)$, \eqref{eigen_eq} implies%
\[
(    \Delta+\lambda)    u(x)=0\,,\qquad x\in\mathbb{R}%
^{3}\backslash\overline\Omega_{\varepsilon}\,.
\]
By the unique continuation property (see, e.g.,  \cite[Theorem XIII.63]{ReSi IV}),  one obtains: $u=0$ in $\mathbb{R}^{3}%
\backslash\overline\Omega_{\varepsilon}$. Henceforth (\ref{eigen_eq}) recasts as%
\[
\begin{cases}
(    \varepsilon^{2}\Delta+\lambda)    u(x)=0\,,&x\in
\Omega_{\varepsilon}\,,\\
u(x)=0\,,& x\in\partial\Omega_{\varepsilon}\,.
\end{cases}
\]
This is a boundary value problem in $\Omega_{\varepsilon}$ for the Helmholtz equation with zero Dirichlet trace; by the uniqueness of its solution, we get $u=0$ in
$\overline\Omega_{\varepsilon}$.
\end{proof}

\begin{lemma}
\label{Lemma_ess}$\sigma_{ess}(    H(\ve))    =[    0,+\infty)    $.
\end{lemma}

\begin{proof} 
We use the resolvent difference formula \eqref{Res_id}. The adjoint
$(    R_{\kappa}(    \varepsilon)    1_{\Omega_{\varepsilon}})  
^{\ast}$ in the Hilbert space $L^{2}(\RE^{3}\!,\ve)    $ is $(    R_{\kappa}(    \varepsilon)  
1_{\Omega_{\varepsilon}})    ^{\ast}=1_{\Omega_{\varepsilon}}R_{\bar{z}%
}(    \varepsilon)    $. By the equivalence \eqref{equi}
 and by the mapping properties of $R_{\kappa}(  
\varepsilon)    $, one gets $(  
R_{\kappa}(    \varepsilon)    1_{\Omega_{\varepsilon}})    ^{\ast}%
\in{\bou}(    L^{2}(  \mathbb{R}^{3})  ,H^{2}(  
\Omega_{\varepsilon})    )    $. By the compact embedding
$H^{2}(    \Omega_{\varepsilon})    \hookrightarrow L^{2}(  
\Omega_{\varepsilon})    $, the operator $(    R_{\kappa}(  
\varepsilon)    1_{\Omega_{\varepsilon}})    ^{\ast}\in{\bou}%
(    L^{2}(  \mathbb{R}^{3}))    $ is compact, which implies the compactness of
$R_{\kappa}(    \varepsilon)    1_{\Omega_{\varepsilon}}\in{\bou}(  
L^{2}(  \mathbb{R}^{3})  )    $ and, hence, the compactness of
the r.h.s. of (\ref{Res_id}). Therefore, the Weyl theorem applies (see, e.g., \cite[Th.
XIII.14]{ReSi IV}) and $\sigma_{ess}(    H(\ve))    =\sigma_{ess}(  -\Delta)=[0,+\infty)$.
\end{proof}

\subsection{The limiting absorption principle} \hfill
\vskip5pt\par\noindent
Here, we prove that there is no singular continuous spectrum, so that  the spectrum of  $H(\varepsilon)$ is purely absolutely continuous. To this end we at first show that a limiting absorption principle holds for   $H(\varepsilon)$. As it is well known, a limiting absorption principle holds for the free Laplacian (see, e.g., \cite[Section 4]{Agm}):\par
For any $\lambda\in\RE\backslash\{0\}$ and for any $\alpha>{\frac12}$, the  limits 
\begin{equation}
\lim_{\delta\searrow0}(-\Delta-(\lambda\pm i\delta))^{-1}
\label{lim1}%
\end{equation}
exist in $\B(L^{2}_{\alpha}(\RE^{3}),H^{2}_{-\alpha}(
\mathbb{R}^{3}) )$; the same hold true in the case $\lambda=0$ whenever $\alpha>3/2$.
Here, $H^{k}_{\alpha}(\mathbb{R}^{3})$, $k\in\mathbb N$, $\alpha\in\RE$, ($L^{2}_{\alpha}(\mathbb{R}^{3})\equiv H^{0}_{\alpha}(\mathbb{R}^{3})$) denotes the weighted Sobolev spaces with weight $w(x)=(1+|x|^{2})^{\alpha}$ (see \cite[Section 2]{Agm} for more details).

As regards $H(\varepsilon)$, the result is of the same kind:

\begin{theorem}
\label{Theorem_LAP} The limits%
\begin{equation}
\lim_{\delta\searrow0}(H(\ve)-(\lambda\pm i\delta))^{-1}
\,, \label{LAP_eps}%
\end{equation}
exist in ${\bou}(    L^{2}_{\alpha}(  \mathbb{R}^{3})
,L^{2}_{-\alpha}(  \mathbb{R}^{3})  )    $ for any $\alpha>1/2$ and any $\lambda\in\mathbb{R}\backslash\{0\}$; the same hold true in the case $\lambda=0$ whenever $\alpha>3/2$.
\end{theorem}
\begin{proof} Introducing, for any $\lambda\in\RE$ and $z\in\CO\backslash[0,+\infty)$, the abbreviated notations
$$
{\mathcal R}^{\pm}_{\lambda}\equiv\lim_{\delta\searrow0}{\mathcal R}_{\lambda\pm i\delta}\,,\qquad
{\mathcal R}_{z}\equiv(-\Delta-z)^{-1}\,,\qquad 
{\mathcal R}_{z}(\ve)\equiv(H(\ve)-z)^{-1}\,,$$
by \cite[Lemma 1, page 170]{ReSi IV}, one has
\begin{equation}
\mathcal{R}_{z}\in \mathscr B(L_{\alpha }^{2}(\mathbb{R}^{3}))\,,
\qquad \alpha \in \mathbb{R}\,.
\label{H1.1}
\end{equation}
Then, by $1_{\Omega _{\varepsilon }}\in {\mathscr B}(L_{\alpha }^{2}(\mathbb{%
R}^{3}))$, by \eqref{H1.1}, by \eqref{res} and \eqref{KR_comp}, 
\begin{equation}
\mathcal{R}_{z}(\varepsilon )\in \mathscr B(L_{\alpha }^{2}(\mathbb{R}%
^{3}))\,,
\qquad \alpha \in 
\mathbb{R}\,.
\label{H1.2}
\end{equation}
By arguing as in the proof of Lemma \ref{Lemma_ess} and by the continuous
embedding $L^{2}(\Omega _{\varepsilon })\hookrightarrow L_{\beta }^{2}(%
\mathbb{R}^{3})$, one gets $\mathcal{R}_{z}(\varepsilon )1_{\Omega
_{\varepsilon }}\in {\mathfrak{S}}_{\infty }(L_{\alpha }^{2}(\mathbb{R}%
^{n}),L_{\beta }^{2}(\mathbb{R}^{3}))$ for any real $\alpha $ and $\beta $.
Hence, in particular, by \eqref{Res_id}, 
\begin{equation}
\mathcal{R}_{z}(\varepsilon )-\mathcal{R}_{z}\in {\mathfrak{S}}_{\infty
}(L^{2}(\mathbb{R}^{n}),L_{\beta }^{2}(\mathbb{R}^{3}))\,,\quad \beta
>2\alpha \,.
\label{H2}
\end{equation}
Furthermore, according to \cite[Corollary 5.7(b)]{BAD}, for all compact
subset $K\subset (0,+\infty )$ there exists a constant $c_{K}>0$ such that,
for $\lambda \in K$ one has 
\begin{equation}
\forall u\in L_{2\alpha }^{2}(\mathbb{R}^{n})\cap \ker (\mathcal{R}_{\lambda
}^{+}-\mathcal{R}_{\lambda }^{-}),\quad \Vert \mathcal{R}_{\lambda }^{\pm
}u\Vert _{L^{2}(\mathbb{R}^{3})}\leq c_{K}\Vert u\Vert _{L_{2\alpha }^{2}(%
\mathbb{R}^{3})}\,.  \label{H3}
\end{equation}
The relations \eqref{H1.1}-\eqref{H3} permit us to apply the abstract
results provided in \cite{Ren1}, where, with respect to the notations there, 
$H_{1}=-\Delta $, $\mathcal{H}_{1}=L^{2}(\mathbb{R}^{3})$, $%
H_{2}=H(\varepsilon )$, $\mathcal{H}_{2}=L^{2}(\mathbb{R}^{3}\!,\varepsilon )
$, $J_{1}$ the identity in $\mathcal{H}_{1}$, $J_{2}:\mathcal{H}%
_{1}\rightarrow \mathcal{H}_{2}$ the multiplication by $b_{\varepsilon
}^{1/2}$, $X=L_{\alpha }^{2}(\mathbb{R}^{3})$ : hypothesis (T1) and (E1) in 
\cite[page 175]{Ren1} corresponds to our \eqref{lim1}, \eqref{H3} and %
\eqref{H2} respectively; then, by \cite[Proposition 4.2]{Ren1}, the latter
imply hypotheses (LAP) and (E) in \cite[page 166]{Ren1}, i.e. \eqref{lim1}
again and 
\begin{equation*}
J_{2}^{\ast }\mathcal{R}_{z}(\varepsilon )J_{2}-\mathcal{R}_{z }\in {\mathfrak{S}}%
_{\infty }(L_{-\alpha }^{2}(\mathbb{R}^{3}),L_{\alpha }^{2}(\mathbb{R}%
^{3}))\,,
\end{equation*}
and hypothesis (T) in \cite[page 168]{Ren1}, a technical variant of %
\eqref{H3}. By \cite[Theorem 3.5]{Ren1}, these last hypotheses, together
with $H(\varepsilon )\geq 0$ and \eqref{H1.1}-\eqref{H1.2} (i.e. hypothesis
(OP) in \cite[page 165]{Ren1}), give the existence of the limits \eqref{LAP_eps} outside $\sigma_{p}(H(\ve))$; the latter is empty by Lemma \ref{sp}.
\end{proof}
\par
From the selfadjointness of $H(\varepsilon)$ on $L^{2}(\RE^{3}\!,\ve)$, we have the orthogonal decomposition%
\[
L^{2}(\RE^{3}\!,\ve)  =(    L^{2}(    \mathbb{R}^{3}\!,\ve)    )    ^{c}\oplus(  L^{2}(\RE^{3}\!,\ve))  ^{pp}\,,
\]
where $(    L^{2}(\RE^{3}\!,\ve)  )    ^{c}$ and $(  
L^{2}(\RE^{3}\!,\ve)  )    ^{pp}$ respectively denote the
continuous and pure point subspaces of $H(\varepsilon)$. Let $P_{\lambda
}(    \varepsilon)    $ be the spectral resolution of the identity
associated with $H(\varepsilon)$. Recall that if $f\in(    L^{2}(\RE^{3}\!,\ve)   )    ^{c}=(    (  L^{2}(\RE^{3}\!,\ve))  ^{pp})    ^{\bot}$, $\lambda\mapsto\left\langle
f,P_{\lambda}(    \varepsilon)    f\right\rangle $ is a continuous
function on $\sigma_{c}(    H(\varepsilon))    $. The continuous
subspace $(    L^{2}(\RE^{3}\!,\ve)  )    ^{c}$ further
decomposes as an orthogonal sum of the absolutely continuous and singular
subspaces of $H(\varepsilon)$%
\[
(    L^{2}(\RE^{3}\!,\ve)  )    ^{c}=(    L^{2}(\RE^{3}\!,\ve)   )    ^{ac}\oplus(    L^{2}(    \mathbb{R}%
^{3})    )    ^{sc}\,.
\]
The absolute continuous part of the spectrum ${\sigma}_{{ac}%
}(    H(\varepsilon))    $ is defined by the condition that, if
$f\in(    L^{2}(\RE^{3}\!,\ve)  )    ^{ac}$, then
$\lambda\mapsto\left\langle f,P_{\lambda}(    \varepsilon)  
f\right\rangle $ is absolutely continuous on $\sigma_{{ac}%
}(    H(\varepsilon))    $. Hence, the absence of singular continuous
spectrum corresponds to the identity%
\[
(    L^{2}(\RE^{3}\!,\ve)  )    ^{ac}=(    (  
L^{2}(\RE^{3}\!,\ve)  )    ^{pp})    ^{\bot}\,.
\]
By Theorem \ref{Theorem_LAP}, we get 

\begin{corollary}
\label{Prop_empty_sc}%
\[%
\begin{array}
[c]{ccc}%
\sigma_{ac}(    H(\varepsilon))    =[  
0,+\infty)    \,, &  & \sigma_{sc}(    H(\ve))    =\varnothing\,.
\end{array}
\]

\end{corollary}

\begin{proof}
Arguing as in \cite[Proof of Theorem 6.1]{Agm}, in order to show that
$\sigma_{c}(H(\varepsilon))=\sigma_{ac}(H(\varepsilon))$, it suffices to show that
$\lambda\mapsto\left\langle f,P_{\lambda}(\varepsilon)  
f\right\rangle $ is absolutely continuous in $\RE\backslash\{0\}$ for any 
$f\in((L^{2}(\RE^{3}\!,\ve))^{pp})^{\bot}$.
Let $[a,b]\subset\RE\backslash\{0\}$ and $f\in L^{2}_{\alpha}
(\mathbb{R}^{3})$; by Theorem \ref{Theorem_LAP}
and by Stone's formula one has
\begin{equation*}
\left\langle f,(P_{b}(\varepsilon )-P_{a}(\varepsilon ))f\right\rangle =%
\frac{1}{2\pi i}\int_{a}^{b}\left\langle f,(\mathcal{R}_{\lambda }^{+}(\varepsilon) -\mathcal{R}_{\lambda }^{-}( \varepsilon)
)f\right\rangle d\lambda \,,
\end{equation*}
where we used the abbreviated notation
$$
\mathcal{R}_{\lambda }^{\pm}(\varepsilon)\equiv\lim_{\delta\searrow0}(H(\ve)-(\lambda\pm i\delta))^{-1}\,.
$$
Hence $\lambda\mapsto\left\langle f,P_{\lambda}(  
\varepsilon)    f\right\rangle $ is continuously differentiable at each point in
$\RE\backslash\{0\} $ for any $f\in L^{2}_{\alpha}(  \mathbb{R}^{3})
$. By the same argument as in \cite[Proof of Theorem 6.1]{Agm}, the property
extends to the whole $(    (  L^{2}(  \mathbb{R}^{3}))  ^{pp})    ^{\bot}$.
\end{proof}
\begin{corollary}\label{cor} $$\sigma(    H(\varepsilon))    =\sigma_{ac}(    H(\varepsilon)) =[0,+\infty)$$ and  $\CO_{+}\ni \kappa\mapsto R_{\kappa}(\ve)$ is an analytic $\bou(L^{2}(\RE^{3}))$-valued map.
\end{corollary}
\begin{proof} By Lemma \ref{sp} and Corollary \ref{Prop_empty_sc}, 
$$
{\sigma}(    H(\varepsilon))=\sigma_{p}(    H(\ve))\cup \sigma_{cont}(    H(\ve))=\sigma_{cont}(    H(\ve))=\sigma_{ac}(    H(\ve))\,.
$$ 
By Lemma \ref{sp} and by $\sigma_{{disc}}(    H(\ve))\subseteq \sigma_{p}(    H(\ve))$, 
$$
{\sigma}(    H(\varepsilon))=\sigma_{{disc}}(    H(\ve))\cup \sigma_{ess}(    H(\ve))=[0,+\infty)
\,.
$$  
Finally, since $H(\varepsilon)$ is a closed operator, the map $\CO_{+}\ni\kappa\mapsto (H(\varepsilon)-\kappa^{2})^{-1}$ is analytic.
\end{proof}

\section{\label{Sec_e}Localization and $\varepsilon$-expansions of the
resonances}
In the following, we use the
identification%
\[
L^{2}(  \mathbb{R}^{3})  \equiv L^{2}( \Omega^{c}_{\ve})    \oplus L^{2}(    \Omega_{\varepsilon})    \,,
\]
provided by the unitary map $u\mapsto    1_{\Omega_{\ve}^{c}}u\oplus 1_{\Omega_{\ve}}$.
 Here and below, given the measurable domain $D\subset\RE^{3}$ we denote by $1_{D}: L^{2}(\RE^{3})\to L^{2}(D)$ the bounded linear operator given by the restriction to $D$; then, the adjoint $1^{*}_{D}:L^{2}(D)\to L^{2}(\RE^{3})$ corresponds to the extension by zero. 
 \par
In such a framework, see \cite[Section 4]{MaPo PDEA 24}, 
$$
1+T(\varepsilon)R_{\kappa}:L^{2}( \Omega^{c}_{\ve})    \oplus L^{2}(    \Omega_{\varepsilon}) \to L^{2}( \Omega^{c}_{\ve})    \oplus L^{2}(    \Omega_{\varepsilon}) 
\,,
$$
$$
1+T(\varepsilon)R_{\kappa}=
\begin{bmatrix}
1 & 0\\
-(    1-\varepsilon^{2})  \kappa^{2}  1_{\Omega_{\varepsilon}}R_{\kappa}%
1^{*}_{\Omega^{c}_{\ve}}&     \varepsilon
^{2}-(    1-\varepsilon^{2})    \kappa^{2}1_{\Omega_{\varepsilon}}R_{\kappa}    1^{*}_{\Omega_{\varepsilon}}%
\end{bmatrix}\,.
$$
Hence, for any $\kappa\in\CO_{+}$, introducing the operator
$$
S_{\kappa}(    \varepsilon)  :=    \varepsilon^{2}-(    1-\varepsilon^{2})  
\kappa^{2}1_{\Omega
_{\varepsilon}}R_{\kappa}   1^{*}_{\Omega_{\varepsilon}}\,,
$$
one has, by \eqref{teoK}, 
$$
S_{\kappa}(    \varepsilon)^{-1}\in\bou(L^{2}(\Omega_{\varepsilon}))\,,\qquad \kappa\in\CO_{+}\,,
$$
and
\begin{equation}
( 1+T(\varepsilon)R_{\kappa})    ^{-1}=%
\begin{bmatrix}
1 & 0\\
-(  1-\varepsilon^{2})    S_{\kappa}(    \varepsilon)^{-1}  
\kappa^{2}1_{\Omega_{\varepsilon}}R_{\kappa}1^{*}_{\Omega_{\varepsilon
}^{c}}    & S_{\kappa}(    \varepsilon)^{-1}  
\end{bmatrix}\,.
\label{KR_id_1}%
\end{equation}
Introducing the unitary dilation operator
$$
U(\ve):L^{2}(\RE^{3})\to L^{2}(\RE^{3})\,,\qquad U(\ve)u(x):=\ve^{3/2}u(\ve x)\,,
$$
one has
$$
U(\ve)\chi_{\Omega_{\ve}}=\chi_{\Omega}U(\ve)\,,\qquad 
U(\ve)\chi_{\Omega^{c}_{\ve}}=\chi_{\Omega^{c}}U(\ve)
$$
and 
$$
U(\ve)R_{\kappa}=\ve^{2}R_{\ve \kappa}U(\ve)\,.
$$
Hence, defining
\begin{equation}
N_{\kappa}:=1_{\Omega}R_{\kappa}1^{*}_{\Omega}:L^{2}(\Omega)\to L^{2}(\Omega)\,,
\label{N_k_def}%
\end{equation}
i.e.,
$$
N_{\kappa}u(x)=\frac1{4\pi}\int_{\Omega}\frac{e^{ik|x-y|}}{|x-y|}\,u(y)\,dy\,,
$$
one has
\begin{equation}
S_{\kappa}(  \varepsilon)  =\varepsilon^{2}1_{\Omega_{\ve}}U^{*}_{\varepsilon}%
1^{*}_{\Omega}M_{\kappa}(  \varepsilon)1_{\Omega}U_{\varepsilon}1^{*}_{\Omega_{\ve}}\,, \label{S_k_id}%
\end{equation}
where 
\begin{equation}
M_{\kappa}(  \varepsilon)  :=1-(  1-\varepsilon^{2})
\kappa^{2}N_{\varepsilon\kappa}\,.\label{Sigma_eps_k_def}%
\end{equation}
Furthermore,

\begin{equation}
S_{\kappa}(  \varepsilon)^{-1}  =\varepsilon^{-2}1_{\Omega_{\ve}}U^{*}_{\varepsilon}%
1^{*}_{\Omega}M_{\kappa}(  \varepsilon)^{-1}1_{\Omega}U_{\varepsilon}1^{*}_{\Omega_{\ve}}\,.\label{S_k_inv_id}
\end{equation}

\begin{lemma}
\label{Lemma_spectrum}The map $\CO\ni\kappa\mapsto M_{\kappa}(  \varepsilon)$ is a $\mathscr{B}(  L^{2}(  \Omega))$-valued analytic map and $\CO\ni\kappa\mapsto M_{\kappa}(  \varepsilon)^{-1}$ is a $\mathscr{B}(  L^{2}(  \Omega))$-valued meromorphic map with poles of finite rank. Such poles are the points $\kappa(\ve)\in\CO\backslash\CO_{+}$ such that $\ker(M_{\kappa(\ve)}(  \varepsilon))\not=\{0\}$.
\end{lemma}
\begin{proof} By Sobolev's inequality (see, e.g., \cite[Example 3, Section IX.4]{ReSi II}), $N_{\kappa}$ is Hilbert-Schmidt and hence compact. Since the analytic map $\kappa\mapsto N_{\kappa}$ has values in the space of compact operators in $\mathscr{B}(  L^{2}(  \Omega))$, the thesis is consequence of the meromorphic Fredholm theory (see, e.g., \cite[Theorem XIII.13]{ReSi IV}.
\end{proof}
\par
The set of resonances of $H(\varepsilon)$ is here defined by%
\begin{equation}
\mathsf{e}(H(\varepsilon)):=\{\kappa_{\circ}^{2}(\varepsilon):\kappa\mapsto
S_{\kappa}(\varepsilon)^{-1}\ \text{has a pole at $\kappa_{\circ}%
(\varepsilon)\in\mathbb{C}_{-}$}\}\,. \label{e_def}%
\end{equation}

\begin{remark}
Since $T(\varepsilon)$ is an additive perturbation with compact support on
$\Omega_{\varepsilon}$, the analogy with the case of Schr\"{o}dinger operators
suggests to define the resonances of $H(\varepsilon)$ as the poles of the
meromorphic extension of the map $z\mapsto 1_{\Omega_{\varepsilon}}(-H(\ve)-z)^{-1}
1^{*}_{\Omega_{\varepsilon}}$ in the
non-physical sheet where $\text{Im}(\sqrt{z})<0$. As well known,
the unperturbed resolvent $z\mapsto1_{\Omega_{\varepsilon}}(-\Delta-z)^{-1}1^{*}_{\Omega_{\varepsilon}}$ has an analytic extension to the whole complex plane;
then, reasoning in terms of $\kappa=\sqrt{z}$ and using the formula
(\ref{res}), we identify $\mathsf{e}(H(\varepsilon))$ with the set of
$\kappa^{2}$ such that $\kappa\in\mathbb{C}_{-}$ is a pole of the meromorphic
extension of $\kappa\mapsto 1_{\Omega_{\varepsilon}}(1+T(\varepsilon
)R_{\kappa})^{-1}1^{*}_{\Omega_{\varepsilon}}$. Using (\ref{KR_id_1}), this
corresponds to the definition (\ref{e_def}); furthermore, our definition of resonance set \eqref{e_def} is in agreement with the one given in \cite[Appendix B]{AGHKH} (see \eqref{e_id} below).
\end{remark}

In view of (\ref{S_k_inv_id}) and of the result of Lemma \ref{Lemma_spectrum},
the set $\mathsf{e}(H(\varepsilon))$ is discrete and the right hand side of
(\ref{e_def}) recasts as%
\[
\mathsf{e}(H(\varepsilon))=\{\kappa_{\circ}^{2}(\varepsilon):\kappa\mapsto
M_{\kappa}(\varepsilon)^{-1}\ \text{has a pole at $\kappa_{\circ}%
(\varepsilon)\in{\mathbb{C}}_{-}$}\}\,,
\]
which is equivalent to%
\begin{equation}
\mathsf{e}(H(\varepsilon))=\{\kappa_{\circ}^{2}(\varepsilon):\ker
(M_{\kappa_{\circ}(\varepsilon)}(\varepsilon))\not =\{0\},\ \kappa_{\circ
}(\varepsilon)\in{\mathbb{C}}_{-}\}\,. \label{e_id}%
\end{equation}
The \emph{Newton potential operator} $N_{0}$ of an open bounded domain
$\Omega\subset\mathbb{R}^{3}$ is the not negative, symmetric  operator $N_{0}$ in $L^{2}(\Omega)$ defined by 
$$
N_{0}:L^{2}(\Omega)\to L^{2}(\Omega)\,,\qquad 
N_{0}u(x):=\frac1{4\pi}\int_{\Omega}\frac{u(y)\,dy}{|x-y|}\,.
$$
By Sobolev's inequality (see, e.g., \cite[Example 3, Section IX.4]{ReSi II}), $N_{0}$ is Hilbert-Schmidt and hence compact. Since $-\Delta N_{0}u=u$, one has $\ker(N_{0})=\{0\}$ 
and $N_{0}^{-1}$ is a unbounded self-adjoint operator.  By the spectral theory of compact symmetric operators (see, e.g., \cite[Section 6]{Jo}), 
$$
\sigma_{disc}(N_{0})=\sigma(N_{0})\backslash\{0\}
$$ 
and the orthonormal sequence $\{e_{n}\}_{1}^{+\infty}$ of eigenvectors corresponding to the discrete spectrum is an orthonormal base of $\ker(N_{0})^{\perp}=L^{2}(\Omega)$. We denote by $\{\lambda_{n}\}_{1}^{+\infty}$, the set of eigenvalues (counting multiplicities) indexed in decreasing order. One has $$\|N_{0}\|=\lambda_{1}\ge\lambda_{2}\ge\dots\ge\lambda_{n}\searrow 0\,. 
$$
Furthermore,
$$
\sigma(N_{0}^{-1})=\sigma_{disc}(N_{0}^{-1})=\{\lambda^{-1}:\lambda\in\sigma_{disc}(N_{0})\}.
$$
Let us notice that, depending on $\Omega$,
the eigenvalues of $N_{0}^{-1}$ may not be simple (whenever $\Omega $ is a ball, see \cite{KalSur}).


Now, we introduce the operator 
\[
M_{\kappa}(  0)  :=1-\kappa^{2}N_{0}\,.
\]
By $M_{\kappa}(  0)  =(N_{0}^{-1}-\kappa^{2})N_{0}$, one has 
$M_{\kappa}(  0)^{-1}=N_{0}^{-1}(N_{0}^{-1}-\kappa^{2})^{-1}\in\bou(L^{2}(\Omega))$ for any $\kappa\in\CO$ such that $\kappa^{2}\in\varrho(N_{0}^{-1})\supset\CO\backslash[\lambda_{1}^{-1},+\infty)$. 
Since $N_{0}$ is compact, $\CO\ni\kappa\mapsto M_{k}(0)^{-1}$ is meromorphic with poles of finite rank at the points $\kappa_{\circ}$ such that $\kappa_{\circ}^{2}\in\sigma(N_{0}^{-1})$.
\begin{lemma}\label{Mk0} For any $\kappa\in\CO$ such that $\kappa^{2}\not\in\sigma(N_{0}^{-1})$, one has
$$
\|M_{\kappa}(  0)^{-1}\|\le\sqrt2\,\, \frac{\max\{\lambda_{1}^{-1},
{\text{\rm Re}(\kappa^{2})}\}}
{\text{\rm dist}(\kappa^{2},\sigma(N_{0}^{-1}))}.
$$
\end{lemma}
\begin{proof} 
Given $\delta\in(0,1)$, let 
\begin{align*}
&\Lambda_{\delta}(\kappa^{2}):=\{\lambda\ge\lambda_{1}^{-1}:|\lambda-\kappa^{2}|\ge\delta\lambda\}\\
=&\{\lambda\ge\lambda_{1}^{-1}:(1-\delta^{2})\lambda^{2}-2\lambda\text{\rm Re}(\kappa^{2})+\text{\rm Re}(\kappa^{2})^{2}+\text{\rm Im}(\kappa^{2})^{2}\ge 0\}\,.
\end{align*}
Hence,
$$\Lambda_{\delta}(\kappa^{2})=[\lambda_{1}^{-1},+\infty)$$
whenever
\be\label{C1}
\text{either $\text{\rm Re}(\kappa^{2})<0$ or $0\le\delta\text{\rm Re}(\kappa^{2})\le (1-\delta^{2})^{1/2}|\text{\rm Im}(\kappa^{2})|$}
\ee
and
\begin{align*}
\Lambda_{\delta}(\kappa^{2})=[\lambda_{1}^{-1},+\infty)&\cap\Bigg(\left\{\lambda\le\frac{\text{\rm Re}(\kappa^{2})-\sqrt{\delta^{2}\text{\rm Re}(\kappa^{2})^{2}-(1-\delta^{2})\text{\rm Im}(\kappa^{2})^{2}}}{1-\delta^{2}}\right\}\\
&\cup\left\{\lambda\ge\frac{\text{\rm Re}(\kappa^{2})+\sqrt{\delta^{2}\text{\rm Re}(\kappa^{2})^{2}-(1-\delta^{2})\text{\rm Im}(\kappa^{2})^{2}}}{1-\delta^{2}}\right\}\Bigg)
\end{align*}
whenever
\be\label{C2}
\delta\text{\rm Re}(\kappa^{2})> (1-\delta^{2})^{1/2}|\text{\rm Im}(\kappa^{2})|
\ee
Then, by
$$
\|M_{\kappa}(0)^{-1}\|=\|N_{0}^{-1}(N_{0}^{-1}-\kappa^{2})^{-1}\|=\sup_{\lambda\in\sigma(N_{0}^{-1})}\,\frac{\lambda}{|\lambda-\kappa^{2}|}\,,
$$
one gets 
\begin{align*}
\|M_{\kappa}(0)^{-1}\|=\sup_{\lambda\in\sigma(N_{0}^{-1})\cap\Lambda_{\delta}(\kappa^{2})}\,\frac{\lambda}{|\lambda-\kappa^{2}|}\le\frac1\delta\,,
\end{align*}
whenever \eqref{C1} holds
and
\begin{align*}
\|M_{\kappa}(0)^{-1}\|=&\max\left\{\sup_{\lambda\in\sigma(N_{0}^{-1})\cap\Lambda_{\delta}(\kappa^{2})}\,\frac{\lambda}{|\lambda-\kappa^{2}|}\ ,
\sup_{\lambda\in\sigma(N_{0}^{-1})\backslash(\Lambda_{\delta}(\kappa^{2})\cap\sigma(N_{0}^{-1}))}\,\frac{\lambda}{|\lambda-\kappa^{2}|}\right\}
\\
\le&\max\left\{\frac1\delta\ ,\frac{\sup\{\sigma(N_{0}^{-1})\backslash(\Lambda_{\delta}(\kappa^{2})\cap\sigma(N_{0}^{-1}))\}}{\text{\rm dist}(\kappa^{2},\sigma(N_{0}^{-1}))}
\right\}
\end{align*}
whenever \eqref{C2} holds.
In the latter case, since 
$$
\{\lambda\in\sigma(N_{0}^{-1}): \delta\lambda\le\text{\rm dist}(\kappa^{2},\sigma(N_{0}^{-1}))\}\subseteq\Lambda_{\delta}(\kappa^{2})\cap\sigma(N_{0}^{-1})\,,
$$
one has
$$
\{\sigma(N_{0}^{-1})\backslash(\Lambda_{\delta}(\kappa^{2})\cap\sigma(N_{0}^{-1}))\subseteq 
\{\lambda\in\sigma(N_{0}^{-1}): \delta\lambda>\text{\rm dist}(\kappa^{2},\sigma(N_{0}^{-1}))\}
$$
and so, 
\begin{align*}
&\|M_{\kappa}(0)^{-1}\|\le \frac{\sup\{\sigma(N_{0}^{-1})\backslash(\Lambda_{\delta}(\kappa^{2})\cap\sigma(N_{0}^{-1}))\}}{\text{\rm dist}(\kappa^{2},\sigma(N_{0}^{-1}))}\\
\le&\,\frac1{\text{\rm dist}(\kappa^{2},\sigma(N_{0}^{-1}))}\,\,\max\left\{\frac1{\lambda_{1}}\ ,
\frac{\text{\rm Re}(\kappa^{2})+\sqrt{\delta^{2}\text{\rm Re}(\kappa^{2})^{2}-(1-\delta^{2})\text{\rm Im}(\kappa^{2})^{2}}}{1-\delta^{2}}\right\}\\
\le&\,\frac1{\text{\rm dist}(\kappa^{2},\sigma(N_{0}^{-1}))}\,\,\max\left\{\frac1{\lambda_{1}}\ ,
\frac{\text{\rm Re}(\kappa^{2})}{1-\delta}\right\}
\,.
\end{align*}
By taking $\delta=1/{\sqrt{2}}$, one gets
$$
\|M_{\kappa}(  0)^{-1}\|\le\begin{cases}\sqrt{2}&\text{\rm Re}(\kappa^{2})\le |\text{\rm Im}(\kappa^{2})|\\
\text{\rm dist}(\kappa^{2},\sigma(N_{0}^{-1}))^{-1}\max\left\{\lambda_{1}^{-1},
{\text{\rm Re}(\kappa^{2})}\right\}&\text{\rm Re}(\kappa^{2})> |\text{\rm Im}(\kappa^{2})|
\end{cases}\quad.
$$
The proof is concluded by noticing that, whenever $\text{\rm Re}(\kappa^{2})\le |\text{\rm Im}(\kappa^{2})|$,
\begin{align*}
\text{\rm dist}(\kappa^{2},\sigma(N_{0}^{-1}))^{-1}\max\left\{\lambda_{1}^{-1},
{\text{\rm Re}(\kappa^{2})}\right\}\
\le&
\begin{cases}1&\text{\rm Re}(\kappa^{2})<0\\
2^{-1/2}&0\le\text{\rm Re}(\kappa^{2})\le\lambda_{1}^{-1},\ \text{\rm Re}(\kappa^{2})\le |\text{\rm Im}(\kappa^{2})|\\
1& \lambda_{1}^{-1}<\text{\rm Re}(\kappa^{2})\le |\text{\rm Im}(\kappa^{2})|
\end{cases}\\
\le &1\,.
\end{align*}
\end{proof}

Given $r>\frac1{\lambda_{1}}$,  let us define
\be\label{rp}
r_{\!\circ}:=\frac12\,\min_{\lambda\in\sigma(N_{0})\cap[1/r,+\infty) }\text{\rm dist}(\lambda^{-1},\sigma(N_{0}^{-1})\backslash\{\lambda^{-1}\})\,,\qquad r_{\!+}:=r+r_{\!\circ}\,.
\ee
Then, for any $r_{\!*}\in(0,r_{\circ}\,]$,
\be\label{cap}
\{z\in\CO: \text{\rm dist}(z,\sigma(N_{0}^{-1}))<r_{\!*}\,\}\cap D_{r_{\!+}}
=
\bigcup_{\lambda\in\sigma(N_{0})\cap[1/r,+\infty)} D_{r_{\!*}}(\lambda^{-1})
\ee
and the discs on the right are disjoint. \par
\begin{theorem}\label{loc} For any  $r>\frac1{\lambda_{1}}$, let $r_{\!\circ}$ and $r_{\!+}$ be defined as in \eqref{rp}. Then, for any $r_{\!*}\in(0,r_{\!\circ}\,]$ and for any $\ve\in(0,1)$ such that 
\be\label{Eve}
\ve\,\sqrt2\,\left({\frac{|\Omega|}{4\pi}}^{\frac12}\,\sqrt{r}_{\!\!+}\,e^{\,\ve\sqrt{r}_{\!\!+}d_{\Omega}}+\ve\lambda_{1}\right)\frac{r_{\!+}^{2}}{r_{\!*}}<1\,,
\ee
one has
\be\label{res-cap}
\mathsf{e}(H(\ve))\cap D_{r_{\!+}}\ \subset \bigcup_{\lambda\in\sigma(N_{0})\cap[1/r,+\infty)} D_{r_{\!*}}(\lambda^{-1})\,.
\ee
\end{theorem}
\begin{proof} By \eqref{cap}, it is enough to show that if $\kappa\in\CO\backslash\CO_{+}$ is such that $|\kappa|^{2}<r_{\!+}$ and $\text{\rm dist}(\kappa^{2},\sigma(N_{0}^{-1}))\ge r_{\!*}$
then $M_{\kappa}(\ve)$ has a bounded inverse. By 
$$
M_{\kappa}(\ve)=M_{\kappa}(0)\big(1+M_{\kappa}(0)^{-1}(M_{\kappa}(\ve)-M_{\kappa}(0))\big)\,,
$$
$M_{\kappa}(\ve)$ has a bounded inverse whenever 
\be\label{MM}
\|M_{\kappa}(0)^{-1}(M_{\kappa}(\ve)-M_{\kappa}(0))\|<1\,.
\ee 
The identity%
\[
N_{\varepsilon\kappa}-N_{0}=\kappa\int_{0}^{\varepsilon}N_{s\kappa}^{(
1)  }\,ds\,,\qquad N_{\kappa}^{(  1)  }u(x):=
\frac{i}{4\pi}\int_{\Omega}
e^{i\kappa|x-y|}\,u(y)\,dy\,, 
\]
yields%
\[
\|N_{\varepsilon\kappa}-N_{0}\|\le|\kappa|\int_{0}^{\varepsilon
}\| N_{s\kappa}^{(  1)  }\|\,ds\le \frac{1}{4\pi}\,|\Omega|^{1/2}\,|\kappa|\int_{0}^{\varepsilon
}{e^{s\,|\text{\rm Im}(\kappa)|\,d_{\Omega}}}ds\le{\frac{|\Omega|}{4\pi}}^{\frac12}\, \ve\,|\kappa|\,e^{\ve\,|\text{\rm Im}(\kappa)|\,d_{\Omega}}
\,.
\]
Hence, by Lemma \ref{Mk0}, 
\begin{align*}
&\|M_{\kappa}(0)^{-1}(M_{\kappa}(\ve)-M_{\kappa}(0))\|\le 
\|M_{\kappa}(0)^{-1}\|\,\|(M_{\kappa}(\ve)-M_{\kappa}(0))\|\\
\le&
|\kappa|^{2}\|M_{\kappa}(0)^{-1}\|\,\|(1-\ve^{2})N_{\ve\kappa}-N_{0}\|
\le
|\kappa|^{2}\|M_{\kappa}(0)^{-1}\|\,\big(\|N_{\ve\kappa}-N_{0}\|+\ve^{2}\|N_{0}\|\big)\\
\le&\ve\,\sqrt2\,|\kappa|^{2}\left({\frac{|\Omega|}{4\pi}}^{\frac12}\,|\kappa|\,e^{\,\ve\,|\text{\rm Im}(\kappa)|\,d_{\Omega}}+\ve\lambda_{1}\right)
\frac{\max\{\lambda_{1}^{-1},
{\text{\rm Re}(\kappa^{2})}\}}{\text{\rm dist}(\kappa^{2},\sigma(N_{0}^{-1}))}\\
\le&\ve\,\sqrt2\,\left({\frac{|\Omega|}{4\pi}}^{\frac12}\,\sqrt{r}_{\!\!+}\,e^{\,\ve\sqrt{r}_{\!\!+}d_{\Omega}}+\ve\lambda_{1}\right)\frac{r_{\!+}^{2}}{r_{\!*}}\,.
\end{align*}
Then, \eqref{Eve} gives \eqref{MM}. 
\end{proof}
\begin{corollary}\label{coro} Given $r>\frac1{\lambda_{1}}$, let  $r_{\!\circ}$ and $r_{\!+}$ be as in \eqref{rp} and define
$$
c_{r}:=\sqrt2\,\left({\frac{|\Omega|}{4\pi}}^{\frac12}\,\sqrt{r}_{\!\!+}\,e^{\,\sqrt{r}_{\!\!+}d_{\Omega}}+\lambda_{1}\right)r_{\!+}^{2}
\,.
$$
Then, for any $\ve\in(0,1)$ such that $\ve\le r_{\!\circ}/c_{r}$,
one has 
\be\label{res-cap-1}
\mathsf{e}(H(\ve))\cap D_{r_{\!+}}\ \subset \bigcup_{\lambda\in\sigma(N_{0})\cap[1/r,+\infty)} D_{\ve_{r}}(\lambda^{-1})\,,\qquad \ve_{r}:=c_{r}\,\ve\,.
\ee
Hence, for any $\ve\le r_{\!\circ}/c_{r}$ and for any $\kappa^{2}_{\circ}(\ve)\in H(\ve)\cap D_{r_{\!+}}$,  there exists an unique $\lambda_{\circ}\in\sigma_{disc}(N_{0})$
such that 
\be\label{ep-r}
|\kappa^{2}_{\circ}(\ve)-\lambda_{\circ}^{-1}|\le c_{r}\,\ve\,.
\ee 
\end{corollary}
\begin{proof} By
$$
\ve\,\sqrt2\,\left({\frac{|\Omega|}{4\pi}}^{\frac12}\,\sqrt{r}_{\!\!+}\,e^{\,\ve\sqrt{r}_{\!\!+}d_{\Omega}}+\ve\lambda_{1}\right)\frac{r_{\!+}^{2}}{c_{r}\ve}<
\sqrt2\,\left({\frac{|\Omega|}{4\pi}}^{\frac12}\,\sqrt{r}_{\!\!+}\,e^{\,\sqrt{r}_{\!\!+}d_{\Omega}}+\lambda_{1}\right)\frac{r_{\!+}^{2}}{c_{r}}=1\,,
$$
the thesis is consequence of Theorem \ref{loc} with $r_{\!*}=c_{r}\ve$. 
\end{proof}

We are now in the position to prove our main result and provide the
$\varepsilon$-expansion for $\kappa_{\circ}^{2}(\varepsilon)$.

\subsection{Proof of Theorem \ref{Theorem_main}} \hfill
\vskip5pt\par\noindent

Let $\lambda_{\circ}\in\sigma_{disc}(N_{0})$ with $\dim(\ker(\lambda_{\circ
}-N_{0}))=m$. By \cite[Theorem 1, Section 10.4.3]{Bau}, there are $m^{\prime
}\leq m$ multivalued functions $z\mapsto\zeta_{j}(z)$ (not necessarily
distinct), such that
\[
\zeta_{j}(0)=\lambda_{\circ}\,,\qquad\ker(N_{z}-\zeta_{j}(z))\not =\{0\}\,.
\]
Such functions have, in a neighborhood of $0\in{\mathbb{C}}$, convergent
Puiseux series expansions
\begin{equation}
\zeta_{j}(z)=\lambda_{\circ}+\sum_{n=1}^{+\infty}\zeta_{j}^{(n)}z^{n/p_{j}%
}\,,\qquad\sum_{j=1}^{m^{\prime}}p_{j}=m\,.\label{bau}%
\end{equation}
Since $z\mapsto N_{iz}$ is self-adjoint on the real axis, by \cite[Theorem 1,
Section 2, Chapter II]{Rell}
, $z\mapsto\zeta_{j}(iz)$ is holomorphic in a neighborhood of $0\in
{\mathbb{C}}$ with a convergent power series expansion
\begin{equation}
\zeta_{j}(iz)=\lambda_{\circ}+\sum_{n=1}^{+\infty}\widetilde{\zeta}%
_{j}^{\,(n)}(iz)^{n}\,.\label{re}%
\end{equation}
By \eqref{bau} and \eqref{re}, one has, whenever $p_{j}>1$,
\[
\zeta_{j}^{(1)}=\lim_{z\rightarrow0}\,(\zeta_{j}(iz)-\lambda_{\circ
})(iz)^{-1/p_{j}}=\lim_{z\rightarrow0}\,(\widetilde{\zeta}_{j}^{\,(1)}%
(iz)^{1-1/p_{j}}+o(z^{1-1/p_{j}}))=0\,.
\]
Iterating this argument, one gets $\zeta_{j}^{(n)}=0$ for any $n\geq1$. Hence,
$m^{\prime}=m$, any $z\mapsto\zeta_{j}(z)$ is single-valued, holomorphic in a
neighborhood of $0\in{\mathbb{C}}$ and \eqref{bau} holds with $p_{j}=1$.

Let $\kappa_{\circ,j}^{2}(\varepsilon)$ denote a resonance close to
$\lambda_{\circ}^{-1}$ according to the previous Corollary. Since
$\ker(M_{\kappa_{\circ}(\varepsilon)}(\varepsilon))\not =\{0\}$,
equivalently,
\[
\ker(N_{\varepsilon\kappa_{\circ,j}(\varepsilon)}-((1-\varepsilon^{2}%
)\kappa_{\circ,j}^{2}(\varepsilon))^{-1})\not =\{0\}\,,
\]
one gets
\[
\zeta_{j}(\varepsilon\kappa_{\circ}(\varepsilon))=((1-\varepsilon^{2}%
)\kappa_{\circ,j}^{2}(\varepsilon))^{-1}\,,
\]
i.e., $\kappa_{\circ,j}(\varepsilon)$ solves the equation
\begin{equation}
f_{j}(\varepsilon,\kappa_{\circ,j}(\varepsilon))=0\,, \label{eq-f}%
\end{equation}
where
\[
f_{j}(\varepsilon,z):=1-(1-\varepsilon^{2})z^{2}\zeta_{j}(\varepsilon z)
\,.
\]
Since
\[
f_{j}(0,\lambda_{\circ}^{-1/2})=1-\lambda_{\circ}^{-1}\zeta(0)=0\,,\qquad
\frac{\partial f_{j}}{\partial z}(0,\lambda_{\circ}^{-1})=-2\lambda_{\circ
}^{-1/2}\zeta(0)=-2\lambda_{\circ}^{1/2}\not =0\,,
\]
by the implicit function theorem for analytic functions, $\varepsilon
\mapsto\kappa_{\circ,j}(\varepsilon)$ is analytic in a neighborhood of zero
and, writing
\[
\kappa_{\circ,j}(\varepsilon)=\lambda_{\circ}^{-1/2}+\sum_{n=1}^{+\infty
}\kappa_{\circ,j}^{(n)}\,\varepsilon^{n}\,,
\]
by
\[
\frac{\partial f_{j}}{\partial\varepsilon}(0,\lambda_{\circ}^{-1/2}%
)+\frac{\partial f_{j}}{\partial z}(0,\lambda_{\circ}^{-1/2})\,\kappa
_{\circ,j}^{(1)}=0\,,
\]
one gets
\begin{equation}
\kappa_{\circ,j}^{(1)}=-\frac{\zeta_{j}^{(1)}}{2\lambda_{\circ}^{2}}\,.
\label{kz}%
\end{equation}
Let us now determine $\zeta_{j}^{(1)}$. Denoting by $e_{j}^{\circ}$, $1\leq
j\leq m$, the orthonormal set of eigenvectors corresponding to $\lambda
_{\circ}$, the eigenvector equation
\[
N_{z}e(z)=\zeta_{j}(z)e(z)\,,\qquad e(0)=e_{j}^{\circ}%
\]
has, by \cite[Theorem 2, Section 10.4.3]{Bau}, $m^{\prime}\leq m$ solutions
$z\mapsto e_{j,j^{\prime}}(z)$, $1\leq j^{\prime}\leq m_{j}$, $\sum
_{j=1}^{m^{\prime}}m_{j}=m$, which are continuous at $0$ and have, in a
neighborhood of $0\in{\mathbb{C}}$, convergent Puiseux series expansions
\begin{equation}
e_{j,j^{\prime}}(z)=e_{j}^{\circ}+\sum_{n=1}^{+\infty}e_{j,j^{\prime}}%
^{(n)}\,z^{n/p_{j}^{\prime}}\,. \label{bau-v}%
\end{equation}
Now, we proceed as above. Since $z\mapsto N_{iz}$ is self-adjoint on the real
axis, by \cite[Theorem 1, Section 2, Chapter II]{Rell}
, $z\mapsto e_{j,j^{\prime}}(iz)$ is holomorphic in a neighborhood of
$0\in{\mathbb{C}}$ with a convergent power series expansion
\begin{equation}
e_{j,j^{\prime}}(iz)=e_{j}^{\circ}+\sum_{n=1}^{+\infty}\widetilde{e}%
_{j,j^{\prime}}^{\,(n)}\,(iz)^{n}\,. \label{re-v}%
\end{equation}
By \eqref{bau-v} and \eqref{re-v}, one has, whenever $p_{j}^{\prime}>1$,
\[
e_{j,j^{\prime}}^{(1)}=\lim_{z\rightarrow0}\,(e_{j,j^{\prime}}(iz)-e_{j}%
^{\circ})(iz)^{-1/p_{j}^{\prime}}=\lim_{z\rightarrow0}\,(\widetilde{e}%
_{j,j^{\prime}}^{\,(1)}\,(iz)^{1-1/p_{j}^{\prime}}+o(z^{1-1/p_{j}^{\prime}%
}))=0\,.
\]
Iterating this argument, one gets $e_{j,j^{\prime}}^{(n)}=0$ for any $n\geq1$.
Hence, any $z\mapsto e_{j,j^{\prime}}(z)$ is holomorphic in a neighborhood of
$0\in{\mathbb{C}}$ and \eqref{bau-v} holds with $p_{j}^{\prime}=1$.

Then, differentiating at $z=0$ the relation
\[
\langle e^{\circ}_{j}, N_{z}e_{j,j^{\prime}}(z)\rangle=\zeta_{j}(z)\langle
e^{\circ}_{j},e_{j,j^{\prime}}(z)\rangle
\]
one gets
\[
\langle e^{\circ}_{j}, N_{0}^{(1)}e^{\circ}_{j}\rangle+\langle e^{\circ}_{j},
N_{0}e_{j,j^{\prime}}^{(1)}\rangle=\zeta_{j}^{(1)}+\lambda_{\circ}\langle
e^{\circ}_{j},e_{j,j^{\prime}}^{(1)}\rangle\,,
\]
where
\[
N_{0}^{(1)}u(x)=\frac{i}{4\pi}\int_{\Omega} u(x)\,dx=\frac{i}{4\pi}\,\langle1,
u\rangle\,.
\]
Hence,
\[
\zeta_{j}^{(1)}=\langle e^{\circ}_{j}, N_{0}^{(1)}e^{\circ}_{j}\rangle
=\frac{i}{4\pi}\,|\langle1, e^{\circ}_{j}\rangle|^{2}\,.
\]

\vskip20pt
\noindent
{\bf Acknowledgements.} The authors gratefully acknowledge the support of the Next Generation EU-Prin project 2022 ''Singular Interactions and Effective Models in Mathematical Physics''.

\end{document}